\documentclass[12pt]{amsart}

\usepackage{pgf,tikz}
\usetikzlibrary{arrows}

\usepackage{amsmath}
\usepackage{amssymb}
\usepackage{latexsym}
\usepackage{enumerate}
\usepackage{amscd}

\usepackage{tikz}
\usepackage[all]{xy}
\usepackage{graphics}
\usepackage[active]{srcltx}

\newcommand{\HH}{\mathrm{H}}

\newcommand{\Z}{\mathbf{Z}}

\newcommand{\R}{\mathbb{R}}
\newcommand{\C}{\mathbf{C}}
\newcommand{\F}{\mathbf{F}}

\newcommand{\bF}{\mathbf{F}}

\newcommand{\mR}{\mathbb{R}}

\newtheorem{theorem}{Theorem}[section]
\newtheorem{proposition}[theorem]{Proposition}

\newtheorem{lemma}[theorem]{Lemma}

\theoremstyle{definition}
\newtheorem{example}[theorem]{Example}
\newtheorem{definition}[theorem]{Definition}
\newtheorem{remark}[theorem]{Remark}

\newcommand{\wmono}{ \ar@{>->}[r]}
\newcommand{\wmonovert}{ \ar@{>->}[d]}
\newcommand{\cof}{ \ar@{^{(}->}[r]}

\begin{document}

\title{Realizing doubles: a conjugation zoo}

\author{Wolfgang Pitsch}
\address{Universitat Aut\`onoma de Barcelona \\ Departament de Matem\`atiques\\
E-08193 Bellaterra, Spain}
\email{pitsch@mat.uab.es}


\author{J\'er\^ome Scherer}
\address{EPFL \\ Mathematics, Station 8 \\
CH-1015 Lausanne, Switzerland}
\email{jerome.scherer@epfl.ch}

\thanks{The authors are partially supported  by FEDER/MEC grant MTM2016-80439-P. The second
author would like to thank the MPI in Bonn for its hospitality in 2009.}


\newcommand{\W}{\mathcal{W}}
\newcommand{\A}{\mathcal{A}}

\begin{abstract}
Conjugation spaces are topological spaces equipped with an involution such that their fixed points have the same
mod $2$ cohomology (as a graded vector space, a ring, and even an unstable algebra) but with
all degrees divided by two, generalizing the classical examples of complex projective spaces
under complex conjugation. Spaces which are constructed from unit balls in complex Euclidean spaces 
are called spherical and are very well understood.
Our aim is twofold. We construct  ``exotic" conjugation spaces and study the realization question: which spaces 
can be realized as real loci, i.e., fixed points of conjugation spaces. We identify obstructions and provide examples of 
spaces and manifolds which cannot be realized as such.
\end{abstract}
\maketitle

\section*{Introduction}
Spaces with an involution, i.e. endowed with an action of the cyclic group of order $2$, abound in nature, and complex conjugation is one 
possible manifestation. Complex vector spaces, or their one point compactifications, complex projective spaces, Grassmannians, etc.
are all equipped with a conjugation action. They share the common feature that the subspace of fixed points is a scaled-down
version by a factor two. 
Hausmann, Holm, and Puppe noticed in their influential article \cite{MR2171799} that this unprecise statement can actually 
be phrased very precisely in terms of a  \emph{conjugation equation} relating the mod $2$ cohomology of the equivariant space 
and that of the fixed points. The main source of examples is given by so-called spherical conjugation spaces, i.e. built from
conjugation cells (unit balls in complex Euclidean spaces $\mathbb C^n$) via equivariant attaching maps.

The first question we address in this note is the existence of ``exotic conjugation spaces", by which me mean that they are not homotopy
equivalent to spherical ones.
We exhibit such exotic spaces, but notice that their exoticity vanishes when completed at~$2$. In fact, in view
of our recent characterization of conjugation spaces, \cite{PRS}, a better
way to understand this is to work in the stable $\mathbf C_2$-equivariant homotopy category. Smashing any conjugation space with $\HH\underline{\F}$, 
the genuine equivariant Eilenberg-Mac Lane spectrum  for constant $\mathbb{Z}/2$ coefficients,  we get a spectrum that splits as a wedge of copies 
of suspensions $\Sigma^{n\rho} \HH\underline{\F}$ where $\rho$ is the regular representation; in this stable setting any conjugation space 
behaves as if it were built from conjugation cells. 

The second 
question is about realization. Which spaces can be realized as
fixed points of a conjugation space (also called real loci)? The answer depends of course on the 
category we work in, whether it is in a smooth or  topological setting. 
In the topological context it is well known (and elementary,
\cite[Section~5]{MR2171799}) that any sphere $S^n$ can be seen as
the real locus of an even dimensional sphere, in fact the representation sphere $S^{2\rho}$. 
So, as a CW-complex, the smallest space which cannot be realized as a real
locus should have at least three cells. Relying on the famous Hopf
invariant one theorem proven by Adams, \cite{MR0141119}, we show
in Theorem~\ref{them:op2non}
that the octonionic projective plane $\mathbb{O}P^2$ is not a real locus.

\medskip

As a CW-complex $\mathbb{O}P^2$ is a three-cell complex, and this
counter-example is minimal from this point of view. It is not minimal however
from the point of view of the dimension. Since any surface, \cite{MR2171799}, and 
any orientable $3$-dimensional manifold is a (smooth) real locus, by work of Olbermann
\cite{MR2425140}, we start looking at $4$-dimensional spaces.
We show that all $4$-dimensional simply connected manifold are realizable as
real loci of $8$-dimensional conjugation \emph{spaces}.

\medskip

\noindent
{\bf Theorem~\ref{thm realisresult}.}
{\it The homotopy type of any simply connected and $4$-dimensional
smooth manifold is realizable as a real locus of some $8$-dimensional conjugation space.}

\medskip

We believe that not all $4$-dimensional manifolds can be realized as a smooth
real locus, but the obstruction would have to be of geometric nature. Homotopical
obstructions, namely the extension by Floyd, \cite{MR0334256}, of Adams' work to
complexes with four cells, help us to exhibit a $10$-dimensional manifold $Y$ which cannot be realized 
as the fixed points of a $20$-dimensional one. This dimension, $10$, is the best bound known at the moment.
Our numerous failed attempts to find a better one explain why this note took so long to see the light.

Let us conclude with an observation about the $5$-dimensional companion $Z$ of Floyd's manifold $Y$,
which looks like a real locus of $Y$ if there were an involution on~$Y$! We prove that there is
a conjugation action \emph{stably}.

\medskip

\noindent
{\bf Proposition~\ref{thm:Floyd}.}
{\it The suspension spectrum of the $5$-dimensional Floyd manifold $Z$ is the real locus of a $\C_2$-action on the suspension spectrum
of the $10$-dimensional Floyd manifold $Y$.}

\medskip

\noindent
{\bf Acknowledgements.} We would like to thank Jean-Claude Hausmann for asking some of the questions which
we answer in this note. We also thank Nick Kuhn for drawing our attention to Floyd's manifolds.

\section{Conjugation Spaces}
In all this work we will denote by $\C_2 = \langle \tau \, \mid \, \tau^2=1 \rangle$ the group with two elements and by $\bF$ the field with two elements. 
In this first section we briefly recall the definition of a conjugation space and introduce the most natural examples, namely the spherical conjugation spaces.
By convention a $\C_2$-space $X$ is a topological space with a $\C_2$-action, i.e. with a specific choice of an involution given by the action of the element $\tau$. 
By analogy with the conjugation action on the complex numbers, the subspace of fixed points $X^{\C_2}$ will be called the \emph{real locus} of $X$. To emphasize the
difference between smooth and topological categories we will speak of a ``smooth real locus'' when working in the smooth category.

Cohomology is always meant unreduced, and with coefficients in $\bF$, the field with two elements. To avoid unnecessarily cumbersome notation, we simply denote by $\HH^\ast(X)$ the cohomology algebra of a space $X$. We start with a naive definition of what the double of a space should be.
We will quickly see that one does not take into account enough structure.

\begin{definition}\label{def:double}
{\rm Let $Y$ be a connected space. A \emph{cohomological double} of $Y$ is a $\C_2$-space $X$ together with 
an additive isomorphism $\kappa \colon \HH^{2\ast}(X) \rightarrow \HH^\ast(X^{\C_2})$
dividing degrees by~$2$.
}
\end{definition}

This definition has almost no mathematical content and does not reflect at all the complex conjugation situation we have alluded to in the introduction: we require some more structured compatibility between the cohomology of $X$ and $X^{\C_2}$.
The inclusion $X^{\C_2} \hookrightarrow X$ induces a map in cohomology, but it is degree preserving. The way to be able to compare
theese spaces is via their Borel cohomology. 

Let $E\C_2$ be the universal space with $\C_2$-action, its space of orbits is the classifying space $B\C_2$, also known as the infinite 
real projective space $\mR P^\infty$. For any $\C_2$-space $X$, consider the Borel construction $X_{h\C_2}= (E\C_2 \times X)/\C_2$, where 
$\C_2$ acts diagonally on the product. Borel cohomology is defined as $\HH^\ast_{\C_2}(X) = \HH^\ast(X_{h\C_2})$. 
The restriction to ordinary cohomology $\rho\colon  \HH^\ast(X_{h\C_2}) \rightarrow \HH^\ast(X)$ is induced by the natural fiber inclusion
$X \hookrightarrow X_{h\C_2}$ for the projection $X_{h\C_2} \twoheadrightarrow \ast_{h\C_2}= B\C_2$. It relates Borel
cohomology with the ordinary cohomology of the space $X$, where we forget the involution.

There is a second important map, namely the restriction to the Borel cohomology of the fixed points 
$r\colon\HH^\ast(X_{h\C_2}) \rightarrow \HH^\ast((X^{\C_2})_{h\C_2})$. Since $\C_2$ acts trivially on $X^{\C_2}$, the Borel construction 
$(X^{\C_2})_{h\C_2}$ is simply $B\C_{2} \times X^{\C_2}$, and the classical K\"unneth theorem tells us that the graded ring $\HH^\ast((X^{\C_2})_{h\C_2})$ is 
isomorphic to $\HH^\ast(X^{\C_2})[u]$, a polynomial ring  in one variable $u$  of cohomological degree $1$ with coefficients in 
the ordinary mod $2$ cohomology ring of~$X^{\C_2}$.




We are ready now for a more structured version of Definition~\ref{def:double}, namely that of a conjugation space.

\begin{definition}\label{def:conjspaceHHP} \cite[Section~3.1]{MR2171799}
{\rm A \emph{conjugation space} is a $\C_2$-space equipped with an $\HH^*$-frame $(\kappa,\sigma)$, i.e. additive maps
\begin{enumerate}
 \item[a)] $\kappa \colon \HH^{2\ast}(X) \rightarrow \HH^\ast(X^{\C_2})$, an isomorphism dividing degrees by~$2$,

\item[b)] $\sigma \colon \HH^{2\ast}(X) \rightarrow \HH^{2\ast}(X_{h\C_2})$, a section
of $\rho \colon \HH^{2\ast}(X_{h\C_2}) \rightarrow \HH^{2\ast}(X)$,
\end{enumerate}
which satisfy the conjugation equation:
\[
r \circ \sigma(x) = \kappa(x) u^m + lt_m
\]
for all $x \in \HH^{2m}(X)$ and all $m \in \mathbf{N}$, where $lt_m$ is a polynomial in the variable $u$ of 
degree strictly less than $m$.}
\end{definition}

\begin{remark}
\label{rem:smoothdef}
{\rm One defines likewise a \emph{conjugation manifold} as a manifold equipped with a smooth $\C_2$-action
turning it into a conjugation space.
}
\end{remark}

\begin{example}[Conjugation spheres]
\label{exa:firstexample}\cite[Example~3.6]{MR2171799}
Consider the field of complex  numbers $\mathbb C$ as a $\C_2$-space via conjugation and the complex euclidean space $\mathbb C^n$
for any integer $n \geq 0$. Its one-point compactification is a $2n$-dimensional sphere on which $\C_2$ acts by reflection through the equator. 
The fixed points is precisely the equator $S^n$. This is trivially seen to be a conjugation space. From the equivariant stable perspective, \cite{PRS}, 
this is equally obvious since the sphere we just described is precisely the representation sphere $S^{n \rho}$ where $\rho$ is the regular representation,
sum of the trivial and the sign representation. As it is now standard, if $V$ is a finite dimensional real representation of $\C_2$, we denote by $S^V$ its one-point compactification.
%
\end{example}

\begin{example}[Spherical spaces]
\label{ex:spherical}\cite[Section~5.2]{MR2171799}
A \emph{spherical} conjugation space is a $\C_2$-space $X$ that is equipped with an exhaustive filtration
$\ast \subset X_0 \subset X_1 \subset \dots$ such that $X_n$ is obtained from $X_{n-1}$ by attaching
conjugation cells (homeomorphic to unit balls in $\mathbb{C}^n$) via equivariant attaching maps defined on the boundary,
which is a representation sphere $S^{n\rho -1}$. Conjugation spheres as defined above are spherical, and so are 
complex projective spaces $\mathbb{C} P^n$.
\end{example}

The conjugation equation is the technical ingredient which makes the concept of conjugation
space so interesting. As the following simple example shows the sole identification of the space of fixed points with 
``half the original space" is not enough to give the nice properties provided by an $\HH^*$-frame, see also
\cite[Example~1]{MR2198191} for a similar example.

\begin{example}\label{exa:secondexample}
Consider the space $X = S^2 \vee S^4$. We define two actions of $\C_2$
in the following way. The first one is by identifying $S^2$ and $S^4$ as compactified $\mathbb{C}$ and $\mathbb{C}^2$
equipped with complex conjugation. This forms indeed a (spherical) conjugation space with $X^{\C_2} = S^1 \vee S^2$.

The second action is trivial on $S^2$ and acts on a point $(x_1, x_2, x_3, x_4)$ in compactified $\R^4$
by $\tau (x_1, x_2, x_3, x_4) = (x_1, -x_2, -x_3, -x_4)$ so that $(S^4)^{\C_2} = S^1$. Here as well
$X^{\C_2} = S^1 \vee S^2$, so that this equivariant space is a double of $S^1 \vee S^2$ as defined in 
Definition~\ref{def:double}, but not a conjugation space. It is not difficult to see that the conjugation equation 
cannot hold, but it is even more direct to look at the genuine equivariant spectrum $\HH \underline{\F} \wedge X$ and observe that it is precisely
$\Sigma^{2} \HH \underline{\F} \vee \Sigma^{3\rho -2} \HH  \underline{\F}$, hence $X$ is not pure (\cite{PRS}).
\end{example}

One important property of an $\HH^*$-frame is
that it is compatible with the product in cohomology, as shown already by
Hausmann, Holm, and Puppe.

\begin{theorem}\cite[Theorem~3.3]{MR2171799}
\label{thm:ring}
Let $X$ be a conjugation space. The morphisms $\kappa$ and $\sigma$ in the $\HH^*$-frame are ring homomorphisms.
\end{theorem}

Even more is true. Franz and Puppe proved in \cite{MR2198191} that $\kappa$ is also compatible with
the action of the mod $2$ Steenrod algebra $\mathcal A$.

\begin{theorem}\cite[Theorem~1.3]{MR2198191}
\label{thm:Steenrod}
Let $X$ be a conjugation space. For any element $x \in \HH^{2n}(X; \bF)$, one has $\kappa (Sq^{2k} x) = Sq^k \kappa(x)$.
\end{theorem}

These two properties follow also from the naturality of the equivariant stable homotopical viewpoint
as we explain in \cite{PRS}.

\begin{remark}\label{rem:2-complexes}
Hambleton and Hausmann observe that any surface is the real locus of a conjugation $4$-manifold, \cite[Introduction]{HHMathZ}.
In particular the real projective plane $\mathbb R P^2$ is the real locus of the complex projective plane $\mathbb CP^2$.
As an equivariant space, the latter is the homotopy cofiber of the Hopf map $S^{2\rho-1} \rightarrow S^\rho$, the fixed points of which
is the multiplication by $2$ on $S^1$.
\end{remark}

Looking at $\mathbb RP^2$ as associated to the presentation $\langle x \mid x^2 \rangle$, we generalize this to $2$-dimensional complexes 
corresponding to presentations of groups where all relations are \emph{squares}. We do not know what happens for arbitrary fundamental groups. 
Given a presentation $G \cong \langle I \mid R \rangle$ we denote by $X_G$ the associated presentation complex $\bigvee_X S^1 \cup \bigvee_R e^2$
with fundamental group~$G$.

\begin{proposition}
\label{prop:relation}
Let $G$ be a group admitting a presentation where all relations are squares. Then $X_G$ is a real locus.
\end{proposition}

\begin{proof}
We construct a simply connected $4$-dimensional conjugation space $X$ as follows.
We start with a wedge of as many $2$-spheres $S^\rho$ as generators in $G$, say $I$, and attach a conjugation $4$-cell for
each relation $w^2$, where $w$ is a word in the free group on $I$. If $w$ has length $k$, the attaching map is the composition
\[
S^{2\rho -1} \xrightarrow{\eta} S^\rho \xrightarrow{p} \bigvee_k S^\rho \xrightarrow{F} \bigvee_I S^\rho
\]
where $p$ pinches the sphere along $k$ meridians and $F$ identifies the $j$-th pinched sphere with the sphere
corresponding to the $j$-th letter in the word $w$. All maps are equivariant and on fixed points they restrict to
\[
S^{1} \xrightarrow{2} S^1 \xrightarrow{p} \bigvee_k S^1 \xrightarrow{f} \bigvee_I S^1
\]
where the description of $f$ is analogous to that of~$F$. This composite map represents precisely $w^2$. Therefore
$X_G$ is the real locus of~$X$. 
\end{proof}

\begin{example}\label{ex:dihedral}
{\rm There is a presentation $\langle x, y \, \mid \, x^2, y^2, (xy)^4 \rangle$ for the dihedral group $D_8$ of order $8$. 
Thus there is a conjugation space whose fixed points has $D_8$ as a
fundamental group.
}
\end{example}

\section{Exotic conjugation spaces }
In this section we present three different recipes for constructing exotic conjugation spaces. By exotic
we mean non-spherical and we wish to stress that these examples predate our work on purity. This
short section serves thus simply as motivation to find a better setting to study the structure of conjugation spaces.
As we briefly indicate in the concluding Remark~\ref{prop 2compete} below this motivated our work with
Ricka in \cite{PRS} where we characterize conjugation spaces stably and equivariantly, including the
exotic examples we present now.

Our first example takes advantage of the fact that $p$-torsion for $p$ odd is not seen by an $\HH^*$-frame.

\begin{example}
{\rm Let $X$ be a simply connected $p$-torsion space, such as a Moore space $M(\Z/p, n)$
with $n \geq 2$. We equip $X$ with the trivial $\C_2$-action. Since $\HH^*(X) \cong \bF$
is concentrated in degree zero, $X$ is a conjugation space. Likewise, any conjugation
space can be ``made exotic'' by adding $p$-torsion, for example by wedging it with a Moore space.
}
\end{example}

There are other spaces which are not seen by ordinary homology, even with coefficients in the integers.
It is possible sometimes to equip such spaces with a non-trivial action which upgrades them to conjugation
spaces.

\begin{example}
{\rm We start from the construction of Berrick and Casacuberta of a universal acyclic space, given
as a ``wedge of gropes" \cite{MR1670384}. Here $X$ is obtained as a telescope of wedges of
circles. Let $X_n = \bigvee_{1}^{2^n} S^1$ for $n \geq 1$ be a wedge of $2^n$ copies of the circle. 
The map $f\colon S^1 \rightarrow S^1 \vee S^1$ is induced by the commutator map $\Z \rightarrow \Z * \Z$ 
on fundamental groups and $f_n\colon X_n \rightarrow X_{n+1}$
is the wedge of $2^n$ copies of $f$. Define the space $X$ as the homotopy colimit of
\[
X_1 \xrightarrow{f_1} X_2  \xrightarrow{f_2} X_3  \xrightarrow{f_3} \dots
\]
It is an acyclic space, that is $\tilde \HH_*(X; \Z) = 0$ and so $\HH^*(X) \cong \bF$. We define the
action of $\C_2$ on $X_1=S^1 \vee S^1$ and likewise on $X_{n+1} = X_n \vee X_n$ 
by requiring that $\tau$ exchanges the two wedge components. The telescope is hence an equivariant 
diagram and the homotopy colimit inherits an action of $\C_2$. The fixed points $X^{\C_2}$ consists in the 
base point only. Thus $X$ is a conjugation space, the $\HH^*$-frame is the trivial one, defined on $\HH^0$
only.}
\end{example}

Our last example is a particular example of the conjugation spheres Olbermann studies very thoroughly
in \cite{MR2728480}. He shows that every $\bF$-homology $3$-sphere is the fixed point set of an involution 
on $S^6$.

\begin{example}
\label{exm Poincare}
{\rm Let $S$ be the Poincar\'e sphere. We let $X = \Sigma^3 S \simeq S^6$ and define the
action of $\C_2$ as follows. If $(t_1, t_2, t_3, s) \in I^3 \times S$ represents a point in $X$,
where $I = [-1, 1]$, then $\tau(t_1, t_2, t_3, s) = (-t_1, -t_2, -t_3, s)$. Thus $X^{\C_2} = S$. The
mod $2$ cohomology of $S$ is an exterior algebra on one generator in degree $3$ and it
is easy to check the conjugation equation.}
\end{example}

\begin{remark}
\label{prop 2compete}
{\rm Bousfield-Kan $2$-completion, \cite{MR0365573}, kills the exoticity of our examples. 
The first $2$ are acyclic, so $X^\wedge_2$ is contractible, and in Example~\ref{exm Poincare} we have $S^\wedge_2 \simeq S^3$. 

In general, let $X$ be a $2$-complete conjugation space of finite type.
The mod $2$ cohomology of $X$ is concentrated in even degree, so that $X$ must be simply connected by the connectivity
Lemma \cite[6.1]{MR0365573}. The Bockstein spectral sequence collapses at the $E_2$-term, which shows that
$\HH_*(X; \Z)$ is torsion free. Thus $X$ is simply connected and torsion free, which means that, not taking into account the
$\C_2$-action, $X$ is equivalent to a $CW$-complex constructed with even dimensional cells.
It makes thus more sense to work with $2$-complete spaces. 
We believe however that the most adequate framework is the $\C_2$-equivariant stable homotopy category and refer to \cite{PRS} 
for a complete characterization in terms of purity in the sense of \cite{MR3505179}. Instead of working in the unstable category and looking at an unstable 
homological localization functor, we realized that the information given in the definition of a conjugation space is stable as it only depends 
on the equivariant spectrum $\HH \underline{\F} \wedge X$.
}
\end{remark}

\section{The Realizability question in dimension 4}
Given a space $Y$ we are looking for a conjugation space $X$ such that $Y \simeq X^{\C_2}$.
We will provide in Section~\ref{sec:nonrealizable} examples of spaces that cannot be realized as the real locus of a 
conjugation space, but our first aim is to present some positive results for spaces $Y$ of dimension four. 
We will show that any simply connected $4$-dimensional manifold can be realized topologically as a real locus,
therefore we will need to understand equivariant attaching maps for $8$-dimensional cells on the $4$-skeleton.
In the next lemma we write $[-,-]^{\C_2}$ for equivariant homotopy classes of maps.

\begin{lemma}\label{lem rest}
The restriction map $[S^{4\rho -1}, S^{2 \rho} \vee S^{2 \rho}]^{\C_2} \rightarrow \pi_3(S^2 \vee S^2)$ induced by taking fixed
points is surjective.
\end{lemma}

\begin{proof}
The group $\pi_3(S^2\vee S^2)$ contains two maps $\eta_i\colon
S^3 \xrightarrow{\eta} S^2 \hookrightarrow S^2 \vee S^2$, where $eta$ is the Hopf map and the index $i$ is $1$ or $2$
depending on which wedge summand inclusion is used.
From the Hilton-Milnor Theorem, \cite{MR0068218}, we know that $\pi_3(S^2\vee S^2)$ is generated
by $\eta_1, \eta_2$, and the Whitehead product  $\omega\colon S^3 \rightarrow S^2 \vee S^2$. By additivity of 
the restriction map, it is enough to show that both types of generators are in its image.

Recall that the Hopf map $\nu\colon S^7 \rightarrow S^4$ can be
constructed as follows. View $S^7$ as the unit sphere in the
quaternionic plane $\mathbb{H}^2$ and endow it with the natural
translation action by the unit quaternions $S^3$. Then the
Hopf map $\nu$ is the quotient map $S^7 \rightarrow S^7/S^3$. Now, the
quaternions can be viewed as $\mathbb{C} \oplus \mathbb{C}j$, and
simultaneous conjugation of both copies of $\mathbb{C}$ gives a
$\C_2$-action on $\mathbb{H}$ and hence on $\mathbb{H}^2$. A direct
computation shows that the induced action of the unit sphere is
compatible with that of the unit quaternions, and that the induced
action on the quotient $S^4$ is the spherical one. This provides in particular 
an equivariant model for the Hopf map $\nu$ in $[S^{4\rho -1}, S^{2 \rho}]$. 
The induced map on fixed point sets is the previous Hopf
map $\eta\colon S^3 \rightarrow S^2$. Moreover, both inclusions
$S^2 \hookrightarrow S^2 \vee S^2$ can be realized as real loci of
the corresponding inclusions $S^{2 \rho} \hookrightarrow S^{2 \rho} \vee S^{2 \rho}$
of spherical conjugation spaces.

It is even simpler to deal with the Whitehead product since
$S^{2 \rho} \times S^{2 \rho}$ is a spherical conjugation space obtained from the
wedge $S^{2 \rho} \vee S^{2 \rho}$ by attaching an $8$-dimensional cell
along the Whitehead product $W = [\iota_1, \iota_2]$, where $\iota_1$
and $\iota_2$ are the wedge inclusions. The fixed point set of this
product is $S^2 \times S^2$, which shows that the restriction of the
Whitehead bracket $W$ is $\omega$.
\end{proof}

We are now ready for our realizability result for $4$-manifolds, and in fact a little
more since we can realize any attaching map of a $4$-cell to a wedge of $2$-spheres,
not only those which yield a CW-complex having the homotopy type of a $4$-manifold. 

\begin{theorem}\label{thm realisresult}
The homotopy type of any simply connected $4$-dimensional smooth
compact manifold is realizable as a real locus of an $8$-dimensional
conjugation space.
\end{theorem}

\begin{proof}
Such a manifold has the homotopy type of a finite wedge of
$2$-spheres with a $4$-cell attached, see for example \cite[proof of Theorem V.1.5]{MR0506372}. 
To construct our homotopy type
as a real locus it is enough to realize it as a wedge of conjugation
$4$-spheres with a conjugation $8$-cell attached. The homotopy type
of the resulting space is determined (as a space with $\C_2$-action)
by the equivariant homotopy class of the attaching map $S^{4\rho-1}
\rightarrow \bigvee S^{2\rho}$. The homotopy type of the real locus
is determined by the (ordinary) homotopy class of the restriction of
this map to the fixed point sets of the source and target spheres.
The theorem is now a consequence of the previous lemma.
\end{proof}

\begin{example}
\label{ex:HP2}
{\rm The equivariant model for the Hopf map $\nu$ in Lem\-ma~\ref{lem rest}
shows that $\mathbb{C} P^2$ is the real locus of the conjugation space $\mathbb{H} P^2$. The $\C_2$-action
on this $8$-dimensional manifold is smooth, which turns $\mathbb{H} P^2$ into a conjugation manifold.
}
\end{example}

\begin{remark}
\label{rem:smooth4}
{\rm We do not know whether every simply connected closed $4$-dimensional manifold is realizable as the real
locus of an $8$-dimensional conjugation \emph{manifold}.
}
\end{remark}

\begin{remark}
\label{rem:Moore}
{\rm Lemma~\ref{lem rest} shows that any attaching map for a $4$-cell on a wedge of $2$-spheres can be realized
equivariantly as the fixed map of an attaching map for a conjugation $8$-cell on a wedge of conjugation $4$-spheres.
Let us generalize this somewhat.
If we discard $p$-torsion for odd primes as discussed in the previous section, the $3$-skeleton of a simply connected
$4$-dimensional CW-complex is a wedge of copies of spheres and Moore spaces $M(\mathbf Z/2^k, 2)$ since the attaching maps on the
$2$-skeleton, a wedge of $2$-spheres, are all of the form $S^2 \xrightarrow{q} S^2 \hookrightarrow \bigvee S^2$ where
$q$ is some integer in $\mathbb Z \cong \pi_2 S^2$. We next attach $4$-cells, the attaching maps represent elements in the
third homotopy group of the $3$-skeleton. By the Hilton-Milnor Theorem again, \cite{MR0068218}, such classes are sums
of elements of three different types:
\begin{enumerate}
\item elements in $\pi_3 S^2 \cong \mathbf Z$, generated by the Hopf map $\eta$;
\item elements in $\pi_3 M(\mathbf Z/2^k, 2) \cong \mathbf Z/2^{k+1}$, \cite[Lemma~1]{MR300280} or \cite[Lemma~2.1]{MR1736922}, generated by the composite
$S^3 \xrightarrow{\eta} S^2 \hookrightarrow M(\mathbf Z/2^k, 2)$, where the last map is the inclusion of the bottom cell;
\item Whitehead products of elements in the second homotopy group of spheres and Moore spaces.
\end{enumerate}
The three types of maps can be realized equivariantly, replacing the Hopf map $\eta$ by the next Hopf map $\nu$, the Moore space $M(\mathbf Z/2^k, 2)$
by $\Sigma^{\rho} (S^\rho \cup_{2^{k-1} \eta} e^{2\rho})$, and Whitehead products by the analogous Whitehead products in a wedge of
$4$-dimensional conjugation spheres and spherical conjugation $3$-cell complexes $\Sigma^{\rho} (S^\rho \cup_{2^{k-1} \eta} e^{2\rho})$. 
}
\end{remark}

\section{Conjugation manifolds with $3$ cells}

In the next section we provide counter examples to the realizability question. We want them to be as small as possible,
first by counting the number of cells, and second by looking at the dimension.
Since a connected $2$-cell complex is a sphere, let us thus move to spaces with
three cells and in order to limit the number of such cell complexes, we restrict ourselves to manifolds. In this
section we present all conjugation manifolds with three cells. We first deal with the problem in the homotopy category
and refine the constructions in the second and third subsections to make them geometrical.

\subsection{Three cell conjugation spaces}
\label{subsec:divalg}
The dimensions of the cells of a manifold with three cells must be $0$, $n$, and $2n$ by Poincar\'e duality, and the
attaching map must have Hopf invariant one. Therefore, $n$ can only be $1$, $2$, $4$, or $8$ by the Hopf
invariant one Theorem proven by Adams, \cite{MR0141119}. These dimensions correspond to the projective spaces
$\R P^2$ over the reals, $\mathbb C P^2$ over the complex numbers, $\mathbb H P^2$ over the quaternions, and $\mathbb O P^2$
over the octonions.

\begin{proposition}
\label{prop:projective}
The projective spaces $\R P^2$, $\mathbb C P^2$, and $\mathbb H P^2$ are realizable as real loci of
conjugation manifolds.
\end{proposition}

\begin{proof}
We have seen in Example~\ref{ex:spherical} that $\mathbb C P^2$ is a spherical conjugation manifold with fixed points $\R P^2$. We have also
encountered $\mathbb H P^2$ in Example~\ref{ex:HP2} as conjugation manifold with fixed points $\mathbb C P^2$. We only need to realize
$\mathbb H P^2$, which will be done by defining a suitable $\C_2$-action on $\mathbb O P^2$.

We work with the same notation as \cite[Appendix~A.1.7]{MR516508}.
As an $\R$-algebra $\mathbb O$ has a basis $(1, e_1, \dots, e_7)$ and $(1, e_1, e_2, e_4)$ generate a $4$-dimensional subalgebra
that is isomorphic to $\mathbb H$. We let $\tau$ act trivially on this subalgebra, and extend it linearly on $\mathbb O$ by changing the sign
on $e_3, e_5, e_6$, and $e_7$. We construct now a model for the Hopf map $\sigma\colon S^{15} \rightarrow S^8$ by considering the unit sphere 
$S(\mathbb O \times \mathbb O)$ in $\mathbf O^2$. We define $\sigma$ by sending a pair $(x,y)$ to $xy^{-1} \in \mathbb O P^1$ if $y$ is non-zero
and to the point at infinity if $y=0$. This map is equivariant and induces on fixed points the Hopf map $\nu\colon S^7 \rightarrow S^4$.
The mapping cone is the Cayley projective plane, a $16$-dimensional manifold $\mathbb O P^2$. It is endowed with an involution
whose fixed points is $\mathbb H P^2$.
\end{proof}

The above does not quite prove the smooth statement, only the homotopical one. To show that the above projective planes are
smooth real loci, we need to find a conjugation action on a geometric model for the doubles. We start with:

\subsection{Normed division algebras and projective lines}
\label{subsec:divalg}
By a theorem of Hurwitz \cite{MR1512117} the only normed division algebras are $\mathbb{R}, \mathbb{C}, \mathbb{H}$ and $\mathbb{O}$, 
the real, complex quaternionic and octonionic algebras. There is an easy way to construct these in an inductive way which builds on Hamilton's construction 
of the complex numbers and is due to Cayley and Dickson, \cite{MR1502549}. Recall that a conjugation on a real algebra is a real inner automorphism, written
as conjugation $\overline{a}$, that is idempotent and anti-multiplicative $\overline{ab}=\overline{b}\, \overline{a}$.

Start with the real numbers $\mathbb{R}$, with its usual operations and define the conjugation operation to be the identity.
Given a normed division algebra $\mathbb{K}$ with conjugation and of dimension $\leq 4$, define a new division 
algebra $\mathbb{L}$ as $\mathbb{K} \times \mathbb{K}$  equipped with a component-wise addition, a multiplication given by 
$(a,b)(c,d) = (ac -d\overline{b},\overline{a}d+cd)$, and a conjugation defined by $\overline{(a,b)} = (\overline{a},-b)$.

This conjugation action is of little use for us as it has at each stage the real numbers $\mathbb{R}$ as fixed points. 
Nevertheless we have  a second $\C_2$-action also defined inductively by setting:
\begin{itemize}
	\item on $\mathbb{R},$ $\tau(a) =a$ and on $\mathbb{C},$ $\tau(a)= \overline{a}$;
	\item if $\tau$ is defined on $\mathbb{K}$, extend it diagonally to $\mathbb{L}$ by $\tau(a,b)= (\tau(a),\tau(b))$.
\end{itemize}

We get then from a straightforward computation:

\begin{lemma}\label{lem:exoticconj}
 Given a normed division algebra $\mathbb{K}$, the map $\tau\colon \mathbb{K} \rightarrow \mathbb{K}$ is a $\mathbb C$-conjugate linear 
 and multiplicative involution. Moreover for any $a \in \mathbb{K}$, the equality $\tau(\overline{a})= \overline{\tau(a)}$ holds. \hfill{\qed}
\end{lemma}
Let us look at the action on the associated projective lines. 

Given a normed division algebra $\mathbb{K}$, define the line through the point $(0, 0) \neq (x,y) \in \mathbb{K}\times \mathbb{K}$ 
to be:
\[
[x,y] = \left\{\begin{array}{l}
	\{ (\lambda(y^{-1}x),\lambda) \ | \ \lambda \in \mathbb{K}^\ast \} \textrm{ if }  y \neq 0\\
	\\
	\{ (\lambda, \lambda(x^{-1}y)) \ | \ \lambda \in \mathbb{K}^\ast \} \textrm{ if }  x \neq 0
	\end{array} \right.
\]

The set of lines $[x,y]$ is the projective line $\mathbb{K}P^1$ and can be given canonically the structure of a smooth manifold. 
For the algebras $\mathbb{R}, \mathbb{C}, \mathbb{H}$ and $\mathbb{O}$ this gives the spheres $S^1, S^2, S^4$ and $S^8$ respectively. 
The involution $\tau$ on the double $\mathbb L$ induces an involution on $\mathbb{L}P^1$  by $\tau([x,y]) = [\tau(x),\tau(y)]$.

\begin{proposition}\label{prop:actprojline}
The $\C_2$-action $\tau$ on $\mathbb{L}P^1$ gives it the structure of a conjugation manifold with fixed point set $(\mathbb{L}P^1)^{\C_2}=\mathbb{K}P^1$.
\end{proposition}
\begin{proof}
The action is given by $\tau[x,y]= [\tau(x),\tau(y)]$. If $x \neq 0$, then $\tau(x) \neq 0$, the condition $\tau([x,y]) = [x,y]$, 
given the description above, translates into:
\[
\exists a \in \mathbb{L}^\ast \quad (a, a(x^{-1}y)) = (1, \tau(x^{-1}y)).
\]
This is equivalent to $a=1$ and $\tau(x^{-1}y) \in \mathbb{K}$.
\end{proof}

\subsection{Projective planes}
\label{subsec:plane}
The case of projective planes is more delicate because of the lack of associativity on $\mathbb{O}$. We define projective spaces via projectors 
on Hermitian spaces, see for example \cite[Appendix~A.6]{MR516508} and discuss only the case of projective planes.

\begin{definition}\label{def:ermatasjordalg}
Given a normed division algebra $\mathbb{K}$, the \emph{Jordan algebra} $\mathfrak{h}_3(\mathbb{K})$ is the space of $3 \times 3$ Hermitian matrices 
with coefficients in $\mathbb{K}$, endowed with the product $a \circ b = \frac{1}{2}(ab +ba)$.
\end{definition}

So an element of $\mathfrak{h}_3(\mathbb{K})$ is a matrix 
$\left(
\begin{matrix}
\alpha & x &  z \\
\overline{x} & \beta & y \\
\overline{z} & \overline{y} & \gamma
\end{matrix}
\right)$,
where $\alpha, \beta, \gamma \in \mathbb{R}$ and $x,y,z \in \mathbb{K}$.

Observe that, because $\tau$ is multiplicative and linear, it is also multiplicative with respect to the Jordan algebra product.
We turn now to the definition of projective spaces in terms of projectors. Recall that a projector is an element such that $p \circ p = p$.

\begin{definition}\label{def:projspaces}
The \emph{projective plane} $\mathbb{K}P^2$ is the subspace of $\mathfrak{h}_3(\mathbb{K})$ consisting of projectors $p$ of trace~$1$.
\end{definition}
 
We conclude from \cite[Theorem~A.6.16]{MR516508} that $\mathbb O P^2$ is a conjugation manifold since the model
given in Definition~\ref{def:projspaces} is homeomorphic with the Cayley projective plane.

\begin{proposition}\label{prop:conjprojspace}
Let $\mathbb{K}$ be a normed division algebra and $\mathbb{L}$ its double. The $\C_2$-action on $\mathfrak{h}_3(\mathbb{L})$ defined by $\tau$ 
coefficient-wise endows $\mathbb{L}P^2$ with the structure of a conjugation manifold with fixed-point set $\mathbb{K}P^2$.	
\end{proposition}
\begin{proof}
Since $\tau$ commutes with conjugation and is compatible with the Jordan product, the fixed point set is $\mathbb{K}P^2$. 
Thus $\tau$ sends projectors to projectors.

Let $r=\dim_\mathbb{R} \mathbb{L}$.  Imposing the condition $\beta = \gamma = 0$ forces our idempotent to be the zero matrix,
a $0$-dimensional cell which is obviously stable by conjugation. Next, the condition $ \gamma =0$ yields projectors of the form
$\left(
\begin{matrix}
\alpha & x &  0 \\
\overline{x} & \beta & 0 \\
0 & 0 & 0
\end{matrix}
\right)$, with $\alpha$ and $x$ such that $\alpha^2 + ||x||^2 =1$.
This subspace is homeomorphic to $S^r = \mathbb{L}P^1$, and is the closure of a conjugation $r$-cell.

Finally, the condition $\gamma \neq 0$, gives an open $2r$-cell (see the explicit calculations in \cite[Prop A.6.16]{MR516508} for 
the octonionic plane, they also work for the other cases). Again, this cell is stable under conjugation. In this way we exhibited
$\mathbb L P^2$ as a spherical conjugation space.
\end{proof}

\begin{remark}
	The construction explained above generalizes to show that the spaces $\mathbb{C}P ^n$ and $\mathbb{H}P ^n$ are spherical conjugation 
spaces for $n \geq 3$ with fixed points the spaces $\mathbb{R}P ^n$ and $\mathbb{C}P ^n$ respectively. The classical CW-decomposition where 
the unique $kr$-cell is obtained by imposing the last $k$-homogeneous coordinates to be non-zero is a decomposition by conjugation cells.  
 \end{remark}
\section{Small non realizable spaces}
\label{sec:nonrealizable}
There exists a three cell complex which cannot be realized as a real locus,
neither as the fixed points of a smooth $\C_2$-action on a manifold, nor as the fixed points
of any conjugation space. The counter-example is the octonionic projective plane $\mathbb{O}P^2$  of course.
We then move on to find smaller counterexamples, dimensionwise.

\begin{theorem}\label{them:op2non}
The octonionic projective plane $\mathbb{O}P^2$ is not a real locus.
\end{theorem}

\begin{proof}
The mod $2$ cohomology of $\mathbb{O}P^2$ is isomorphic to
$\mathbf{F}_2[x]/(x^3)$ with $x$ in degree~$8$. If it where the real
locus of a conjugation space $X$, then, by the compatibility of $\HH^*$-frame with the ring structure, 
Theorem~\ref{thm:ring},  the space $X$ would have the 
same cohomology but with a generator in degree $16$ which is impossible
by Adams' solution to the Hopf invariant one problem, \cite{MR0141119}.
\end{proof}

The octonionic projective plane is a minimal conjugation space, in the sense that it is built with the minimal number of cells, namely three. 
However, it is a $16$-dimensional manifold.
Relying on Floyd's work \cite{MR0334256}, we exhibit a smaller counter-example from the point of view of its dimension. There
exists a $10$-dimensional manifold $Y$ with four cells which cannot be realized as a real locus. 

In his work Floyd proves that there are only two non-trivial unoriented cobordism classes that contain 
manifolds with four cells. They occur in dimensions $5$ and $10$,  \cite[Theorem~3.1]{MR0334256}. 
The larger one is the one we are interested in.

\begin{example}\label{ex:FloydManifold}
{\rm The $10$-dimensional Floyd manifold is constructed as follows.
One starts with $S^3 \times_{S^1} \mathbb {H}P^2$. The action of the circle $S^1 \subset \mathbb C$  on $S^3 \subset \mathbb C \oplus \mathbb C$ 
is given by multiplication on each factor, and on $ \mathbb {H}P^2$ it is induced by the the action on $\mathbb{H} = \mathbb C \oplus j \mathbb C$ given 
by $\alpha \cdot(z +jw) = z + j \alpha w$. On $\mathbb{H} P^2$ there is a point fixed under $\C_2$ and  $S^1$, for instance the point 
with homogeneous quaternionic coordinates $[1:1:1]$. We therefore get an embedded sphere $S^2$ in  $S^3 \times_{S^1} \mathbb {H}P^2$
that splits the fibration. Floyd proceeds, \cite[Lemma~3.5]{MR0334256}, by surgering out this sphere and obtains a manifold $Y$, whose mod $2$ cohomology is four 
dimensional, on classes $1, e_4, e_6$, and $e_{10} = e_4 \cdot e_6$.
}
\end{example}

\begin{theorem}\label{thm:Floyd}
The $10$-dimensional manifold Floyd $Y$ is not a real locus.
\end{theorem}

\begin{proof}
The proof uses the compatibility of the frame with the Steenrod algebra, Theorem~\ref{thm:Steenrod}. 
The unstable module structure on $\HH^\ast Y$ is completely described by the fact that $Sq^2 e_4 = e_6$ and $Sq^4 e_6 = e_{10}$.
If $Y$ were the real locus of a $20$-dimensional conjugation space, its cohomology would then be isomorphic to $\Phi \HH^\ast Y$,
the double of this unstable module in the sense of \cite[Section~1.7]{MR1282727}.
The claim follows from the fact that no $20$-dimensional manifold with the appropriate mod $2$ cohomology, as a module over the 
Steenrod algebra, can exist, \cite[Lemma~3.4]{MR0334256}.
\end{proof}

\begin{remark}\label{rem:Floyd}
{\rm
This $10$-dimensional manifold holds today the record from the point of view of the dimension.
We do not know of any space or manifold of lower dimension which cannot be realized as
real locus either in the topological or the smooth category. In order to make further progress
on this question, it seems necessary to find obstructions which involve the $\HH^*$-frame
or purely geometrical ingredients.

However we believe that the $5$-dimensional Floyd manifold $Z$ cannot be realized as a real locus. The construction
of this manifold parallels that of $Y$, starting from $S^1 \times_{\C_2} \mathbb CP^2$ before surgery. The mod $2$
cohomology is four dimensional on classes $1, f_2, f_3$ and $f_5 = f_2 \cdot f_3$ with Steenrod operations $Sq^1 f_2 = f_3$
and $Sq^2 f_3 = f_5$ connecting the generators. 

Hence the manifold $Z$ can be doubled in the sense of Definition~\ref{def:double}, and the doubling $\kappa$ is actually
compatible both with the ring structure and the unstable module structure, \cite[Lemma~3.5]{MR0334256}. Why do we believe 
that $Z$ cannot be realized as the fixed points $Y^{\C_2}$ for some $\C_2$-action on the $10$-dimensional Floyd manifold?
It is true that $S^1 \times_{S^0} \mathbb {C}P^2$ is the real locus of $S^3 \times_{S^1} \mathbb {H}P^2$, but
the surgery obstruction does not vanish equivariantly.
}
\end{remark}

Let us conclude this section by observing that stably there is no obstruction. We could define formally what a conjugation spectrum is, but in view of
our needs (and our work in \cite{PRS}) we will only use pure $\C_2$-equivariant spectra, i.e. cellular spectra constructed from representation cells
corresponding to multiples of the regular representation $\rho$.

\begin{proposition}\label{thm:Floyd}
The suspension spectrum of the $5$-dimensional Floyd manifold $Z$ is the real locus of a $\C_2$-action on the suspension spectrum
of the $10$-dimensional Floyd manifold $Y$.
\end{proposition}

\begin{proof}
It follows from the computations of Araki and Iriye, \cite{MR656233}, see also Dugger and Isaksen's \cite[Table~1]{MR3652813},
that $\pi_{3 \rho -2}^S = 0$. In particular, if we call $\tilde \eta \in \pi_{\rho -1}^S$ the equivariant lift of the Hopf map $\eta$ and $\tilde \nu \in \pi_{2 \rho -1}^S$
that for the Hopf map $\nu$, we conclude that the composite $\Sigma \tilde \eta \circ \Sigma^\rho \tilde \nu$ is null, hence $\Sigma^\rho \tilde \nu$ factorizes non-trivially
through $g\colon S^{3 \rho -1} \rightarrow S^0 \cup_{\tilde \eta} e^{\rho}$ in two different ways, parametrized by $\pi_{3\rho -1}^S \cong \mathbb Z/2$. 
By using any of them as an attaching map we get a stable three cell complex $S^0 \cup e^{\rho} \cup_g e^{3\rho}$. 

Forgetting the $\C_2$-action this  must be the suspension spectrum (up to $\Sigma^{2\rho}$ suspension to get the dimension right) 
of the $10$-dimensional Floyd manifold  for the following reason. Non-equivariantly there is no indeterminacy in constructing
$g$ since both $\pi_4^S$ and $\pi_5^S$ are trivial. Hence the map $g$ is, forgetting the action, the only non-trivial attaching map for a $6$-cell on
$\Sigma^{-2} \mathbb C P^2$. The presence of the Steenrod square $Sq^4$ in the Floyd manifold shows that the attaching map in the suspension
spectrum is indeed non-trivial.

Since the fixed points of the Hopf map $\tilde \eta$ is the degree $2$ map, the fixed point spectrum is a complex $S^0 \cup_2 e^1 \cup e^3$.
The attaching map for the $3$-cell is understood just as above and is parametrized by $\pi_2^S \cong \mathbb Z/2$ (its is also well-known
that $\pi_2^S \mathbb R P^2 \cong \mathbb Z/2$). The non-zero one yields a non-trivial $Sq^2$ in cohomology.
\end{proof}

In Remark~\ref{rem:Moore} we explained how to realize simply-connected $4$-dimensional $2$-local complexes as real loci
of $3$-connected $8$-dimensional conjugation spaces. It seems difficult to deal with arbitrary $5$-dimensional spaces. In fact even
the $5$-dimensional Floyd manifold seems out of reach from this viewpoint since we would have to find out if the attaching map of the top
cell, representing an element in $\pi_4 \Sigma \mathbb RP^2$, can be realized equivariantly in $[S^{4 \rho -1}, \Sigma^\rho \mathbb CP^2]^{\C_2}$.
However even the explicit description of the generator of $\pi_4 \Sigma \mathbb RP^2 \cong \mathbf Z/4$, given by Wu in 
\cite[Proposition~6.5]{MR1955357}, does not admit an obvious ``doubling''.

\section{Doubling cobordisms and structure cobordisms}
As stated in Theorem~\ref{thm:Steenrod}, in a conjugation space $X$ the ring homomorphism is in fact a morphism of unstable algebras, 
up to a readjustment in degrees. As an immediate consequence for manifolds we proved in \cite[Theorem~A.1]{MR3082744} that the Wu 
classes  and the Stiefel-Whitney classes correspond via the doubling isomorphism~$\kappa_0$. In particular the non-equivariant unoriented 
cobordism class of a conjugation manifold is determined by that of its fixed points. One could hope to find a number $n$ such that the cobordism ring cannot contain \emph{any} conjugation manifold in even dimension $2n$, except
for the zero cobordant ones, as this would show that no non-zero cobordant $n$-dimensional manifold could be a real locus.
This strategy to find non-realizable manifolds (as real loci) turns out to be hopeless as we explain next.

\begin{theorem}\label{thm:conjgeneratecobord}
Every non-equivariant and non-oriented cobordism class contains the fixed locus of a conjugation manifold.
\end{theorem}

\begin{proof}
Thom showed that real projective spaces $\mathbb R P^{2n}$ provide generators in even dimensions in \cite{MR61823}. We have already
seen that they are smooth real loci of the corresponding complex projective spaces $\mathbb C P^{2n}$. Dold
constructed the famous ``Dold manifolds'' in \cite{MR79269} to complete the set of generators in odd dimensions. They are
orbit spaces $P(m, n) = S^m \times_{\C_2} \mathbb C P^n$ where the action is by the antipodal map on the sphere and by
conjugation on the complex projective space. 

These as well are smooth real loci since there is a general construction that is analogous to what Floyd described for doubling
$S^1 \times_{\C_2} \mathbb CP^2$, see Remark~\ref{rem:Floyd}. Define $DP(m, n) = S^{2m+1} \times_{S^1} \mathbb HP^n$, where
$S^1$ acts on the unit sphere $S^{2m+1}$ in $\mathbb C^{m+1}$ by mutliplication on each component and on each quaternionic component
$v + j w \in \mathbb H$ (where $v, w \in \mathbb C$), by multiplication on $w$ only. The conjugation action on $\mathbb C^{m+1}$ and
on $v$ and $w$ is compatible with the circle action, so that $DP(m, n)$ inherits an action of~$\C_2$. The fixed points are $P(m, n)$.
An application of \cite[Proposition~5.3]{MR2171799} shows that $DP(m, n)$ is a conjugation manifold. Finally observe that the generating 
operations of the cobordism ring: connected sum, disjoint union, and cartesian product preserve conjugation spaces (cf. \cite[Section 4]{MR2171799}).
\end{proof}

\bibliographystyle{plain}\label{biblography}

\end{document}